\documentclass[a4paper,12pt,intlimits,oneside]{amsart}

\usepackage{enumerate}
\usepackage{amsfonts,amsmath}

\usepackage{latexsym,amssymb}

\usepackage{mathrsfs}

\newcommand{\comment}[1]{}

\begin{document}

\title{On weak$^*$-convergence in $H^1_L(\mathbb R^d)$}         

\author{Luong Dang Ky}


\begin{abstract}
Let $L= -\Delta+ V$ be a Schr\"odinger operator on $\mathbb R^d$, $d\geq 3$, where $V$ is a nonnegative function, $V\ne 0$, and belongs to the reverse H\"older class $RH_{d/2}$. In this paper, we prove a version of the classical theorem of Jones and Journ\'e on weak$^*$-convergence in  the Hardy space $H^1_L(\mathbb R^d)$.
\end{abstract}

\maketitle
\newtheorem{theorem}{Theorem}[section]
\newtheorem{lemma}{Lemma}[section]
\newtheorem{proposition}{Proposition}[section]
\newtheorem{remark}{Remark}[section]
\newtheorem{corollary}{Corollary}[section]
\newtheorem{definition}{Definition}[section]
\newtheorem{example}{Example}[section]
\numberwithin{equation}{section}
\newtheorem{Theorem}{Theorem}[section]
\newtheorem{Lemma}{Lemma}[section]
\newtheorem{Proposition}{Proposition}[section]
\newtheorem{Remark}{Remark}[section]
\newtheorem{Corollary}{Corollary}[section]
\newtheorem{Definition}{Definition}[section]
\newtheorem{Example}{Example}[section]
\newtheorem*{theoremjj}{Theorem J-J}
\newtheorem*{theorema}{Theorem A}
\newtheorem*{theoremb}{Theorem B}
\newtheorem*{theoremc}{Theorem C}

\section{Introduction}

A famous and classical result of Fefferman \cite{Fef} states that the John-Nirenberg space $BMO(\mathbb R^d)$ is the dual of the Hardy space $H^1(\mathbb R^d)$. It is also well-known that $H^1(\mathbb R^d)$ is one of the few examples of separable, nonreflexive Banach space which is a dual space. In fact,  let $VMO(\mathbb R^d)$ denote the closure of the space $C_c^\infty(\mathbb R^d)$ in $BMO(\mathbb R^d)$, where $C^\infty_c(\mathbb R^d)$ is the set of $C^\infty$-functions with compact support, Coifman and Weiss showed in \cite{CW} that $H^1(\mathbb R^d)$ is the dual space of $VMO(\mathbb R^d)$, which gives to $H^1(\mathbb R^d)$ a richer structure than $L^1(\mathbb R^d)$. For example, the classical Riesz transforms $\nabla (-\Delta)^{-1/2}$ are not bounded on $L^1(\mathbb R^d)$, but are bounded on  $H^1(\mathbb R^d)$. In addition, the weak$^*$-convergence is true in $H^1(\mathbb R^d)$, which is useful  in the application of Hardy spaces to compensated compactness (see \cite{CLMS}). More precisely, in \cite{JJ}, Jones and Journ\'e proved the following.

\begin{theoremjj}
Suppose that $\{f_j\}_{j\geq 1}$ is a bounded sequence in $H^1(\mathbb R^d)$, and that $f_j(x)\to f(x)$ for almost every $x\in\mathbb R^d$. Then, $f\in H^1(\mathbb R^d)$ and $\{f_j\}_{j\geq 1}$ weak$^*$-converges to $f$, that is, for every $\varphi\in VMO(\mathbb R^d)$, we have
$$\lim_{j\to\infty} \int_{\mathbb R^d} f_j(x) \varphi(x)dx = \int_{\mathbb R^d} f(x) \varphi(x) dx.$$
\end{theoremjj}

The aim of this paper is to prove an analogous version of the above theorem in the setting of function spaces associated with Schr\"odinger operators. 

Let $L= -\Delta+ V$ be a Schr\"odinger differential operator on $\mathbb R^d$, $d\geq 3$, where $V$ is a nonnegative potential, $V\ne 0$, and belongs to the reverse H\"older class $RH_{d/2}$. In the recent years, there is an increasing interest on the study of  the problems of harmonic analysis associated with  these operators, see for example \cite{DDSTY, DGMTZ, DZ, Ky2, LL, Sh, YZ2}. In \cite{DZ}, Dziuba\'nski and Zienkiewicz considered the Hardy space $H^1_L(\mathbb R^d)$ as the set of functions $f\in L^1(\mathbb R^d)$ such that $\|f\|_{H^1_L}:=\|\mathcal M_Lf\|_{L^1}<\infty$, where $\mathcal M_L f(x): = \sup_{t>0}|e^{-tL}f(x)|$. There, they characterized $H^1_L(\mathbb R^d)$ in terms of atomic decomposition and in terms of the Riesz transforms associated with $L$. Later, in \cite{DGMTZ}, Dziuba\'nski et al. introduced a $BMO$-type space $BMO_L(\mathbb R^d)$  associated with $L$, and established the duality between $H^1_L(\mathbb R^d)$ and $BMO_L(\mathbb R^d)$. Recently, Deng et al.  \cite{DDSTY} introduced and developed new $VMO$-type function spaces   $VMO_A(\mathbb R^d)$ associated with some  operators $A$  which have a bounded holomorphic functional calculus on $L^2(\mathbb R^d)$. When $A\equiv L$, their space  $VMO_L(\mathbb R^d)$ is just the set of all functions $f$ in $BMO_L(\mathbb R^d)$ such that $\gamma_1(f)= \gamma_2(f)= \gamma_3(f)=0$, where
$$\gamma_1(f)=\lim\limits_{r\to 0}\left(\sup\limits_{x\in\mathbb R^d, t\leq r}\Big(\frac{1}{|B(x,t)|}\int_{B(x,t)}|f(y)- e^{-tL} f(y)|^2 dy\Big)^{1/2}\right),$$
$$\gamma_2(f)=\lim\limits_{R\to \infty}\left(\sup\limits_{x\in\mathbb R^d, t\geq R}\Big(\frac{1}{|B(x,t)|}\int_{B(x,t)}|f(y)- e^{-tL} f(y)|^2 dy\Big)^{1/2}\right),$$
$$\gamma_3(f)=\lim\limits_{R\to \infty}\left(\sup\limits_{B(x,t)\cap B(0,R)=\emptyset}\Big(\frac{1}{|B(x,t)|}\int_{B(x,t)}|f(y)- e^{-tL} f(y)|^2 dy\Big)^{1/2}\right).$$
The authors  in \cite{DDSTY} further showed that  $H^1_L(\mathbb R^d)$ is in fact the dual of $VMO_L(\mathbb R^d)$, which allows us to study  the weak$^*$-convergence  in  $H^1_L(\mathbb R^d)$. This is useful in the study of the Hardy estimates for commutators of singular integral operators related to $L$, see for example Theorem 7.1 and Theorem 7.3 of \cite{Ky2}.

Our main result is the following theorem.

\begin{theorem}\label{the main theorem}
Suppose that $\{f_j\}_{j\geq 1}$ is a bounded sequence in $H^1_L(\mathbb R^d)$, and that $f_j(x)\to f(x)$ for almost every $x\in\mathbb R^d$. Then, $f\in H^1_L(\mathbb R^d)$ and $\{f_j\}_{j\geq 1}$ weak$^*$-converges to $f$, that is, for every $\varphi\in VMO_L(\mathbb R^d)$, we have
$$\lim_{j\to\infty} \int_{\mathbb R^d} f_j(x) \varphi(x)dx = \int_{\mathbb R^d} f(x) \varphi(x) dx.$$
\end{theorem}

Throughout the whole paper, $C$ denotes a positive geometric constant which is independent of the main parameters, but may change from line to line.   In $\mathbb R^d$, we denote by $B=B(x,r)$ an open ball with center $x$ and radius $r>0$. For any measurable set $E$, we denote  by $|E|$ its  Lebesgue measure.

The paper is organized as follows. In Section 2, we present some notations and preliminary results. Section 3 is devoted to the proof of Theorem \ref{the main theorem}. In the last section, we prove that $C^\infty_c(\mathbb R^d)$ is dense in the space  $VMO_L(\mathbb R^d)$.

{\bf Acknowledgements.} The author would like to thank  Aline Bonami and  Sandrine Grellier for many helpful suggestions and discussions.

\section{Some preliminaries and notations}\label{Some preliminaries and notations}

In this paper, we consider the Schr\"odinger differential operator
$$L= -\Delta+ V$$
 on $\mathbb R^d$, $d\geq 3$, where $V$ is a nonnegative potential, $V\ne 0$.  As in the works of Dziuba\'nski et al \cite{DGMTZ, DZ}, we always assume that $V$ belongs to the reverse H\"older class $RH_{d/2}$. Recall that a nonnegative locally integrable function $V$ is said to belong to a reverse H\"older class $RH_q$, $1<q<\infty$, if  there exists a constant $C>0$ such that for every ball $B\subset \mathbb R^d$, 
$$\Big(\frac{1}{|B|}\int_B (V(x))^q dx\Big)^{1/q}\leq \frac{C}{|B|}\int_B V(x) dx.$$

Let $\{T_t\}_{t>0}$ be the semigroup generated by $L$ and $T_t(x,y)$ be their kernels. Namely,
$$T_t f(x)=e^{-t L}f(x)=\int_{\mathbb R^d} T_t(x,y)f(y)dy,\quad f\in L^2(\mathbb R^d),\quad t>0.$$
Since $V$ is nonnegative, the Feynman-Kac formula implies that
\begin{equation}\label{the Feynman-Kac formula}
0\leq T_t(x,y)\leq \frac{1}{(4\pi t)^{d/2}} e^{-\frac{|x-y|^2}{4t}}.
\end{equation}

According to \cite{DZ}, the space $H^1_L(\mathbb R^d)$ is defined as the completion of 
$$\{f\in L^2(\mathbb R^d): \mathcal M_Lf\in L^1(\mathbb R^d) \}$$
in the norm 
$$\|f\|_{H^1_L}:= \|\mathcal M_L f\|_{L^1},$$
 where $\mathcal M_L f(x):= \sup_{t>0}|T_t f(x)|$ for all $x\in \mathbb R^d$.

In \cite{DGMTZ} it was shown that the dual space of $H^1_{L}(\mathbb R^d)$ can be identified with the space $BMO_L(\mathbb R^d)$ which consists of all functions $f\in BMO(\mathbb R^d)$ with
\begin{equation}
\|f\|_{BMO_L} := \|f\|_{BMO}+ \sup_{\rho(x)\leq r}\frac{1}{|B(x,r)|}\int_{B(x,r)}|f(y)|dy<\infty,
\end{equation}
where $\rho$ is  the auxiliary function defined as in \cite{Sh}, that is,
\begin{equation}
\rho(x)= \sup\Big\{r>0: \frac{1}{r^{d-2}}\int_{B(x,r)} V(y)dy\leq 1\Big\},
\end{equation}
$x\in \mathbb R^d$. Clearly, $0<\rho(x)<\infty$ for all $x\in \mathbb R^d$, and thus $\mathbb R^d=\bigcup_{n\in\mathbb Z}\mathcal B_n$, where the sets $\mathcal B_n$ are defined by
\begin{equation}
\mathcal B_n= \{x\in \mathbb R^d: 2^{-(n+1)/2}< \rho(x)\leq 2^{-n/2}\}.
\end{equation}

The following fundamental property of the function $\rho$ is due to Shen \cite{Sh}.

\begin{Proposition} [see \cite{Sh}, Lemma 1.4] \label{Shen, Lemma 1.4}
There exist $C_0>1$ and $k_0\geq 1$ such that for all $x,y\in\mathbb R^d$,
$$C_0^{-1}\rho(x) \Big(1+ \frac{|x-y|}{\rho(x)}\Big)^{-k_0}\leq \rho(y)\leq C_0 \rho(x) \Big(1+ \frac{|x-y|}{\rho(x)}\Big)^{\frac{k_0}{k_0+1}}.$$
\end{Proposition}

Let $VMO_L(\mathbb R^d)$ be the subspace of $BMO_L(\mathbb R^d)$ consisting of those functions $f$ satisfying $\gamma_1(f)= \gamma_2(f)= \gamma_3(f)=0$, where
$$\gamma_1(f)=\lim\limits_{r\to 0}\left(\sup\limits_{x\in\mathbb R^d, t\leq r}\Big(\frac{1}{|B(x,t)|}\int_{B(x,t)}|f(y)- e^{-tL} f(y)|^2 dy\Big)^{1/2}\right),$$
$$\gamma_2(f)=\lim\limits_{R\to \infty}\left(\sup\limits_{x\in\mathbb R^d, t\geq R}\Big(\frac{1}{|B(x,t)|}\int_{B(x,t)}|f(y)- e^{-tL} f(y)|^2 dy\Big)^{1/2}\right),$$
$$\gamma_3(f)=\lim\limits_{R\to \infty}\left(\sup\limits_{B(x,t)\cap B(0,R)=\emptyset}\Big(\frac{1}{|B(x,t)|}\int_{B(x,t)}|f(y)- e^{-tL} f(y)|^2 dy\Big)^{1/2}\right).$$
In \cite{DDSTY} it was shown  that  $H^1_L(\mathbb R^d)$ is the dual space of  $VMO_L(\mathbb R^d)$.

In the sequel, we denote by $\mathcal C_L$ the $L$-constant
$$\mathcal C_L= 8. 9^{k_0}C_0$$
where $k_0$ and $C_0$ are defined as in Proposition \ref{Shen, Lemma 1.4}. 

Following Dziuba\'nski and Zienkiewicz \cite{DZ}, we define atoms as follows.

\begin{Definition}
Given $1<q\leq \infty$. A function $a$ is called a $(H^1_L,q)$-atom related to the ball $ B(x_0,r)$ if $r\leq  \mathcal C_L \rho(x_0) $ and

i) supp $a\subset B(x_0,r)$,

ii) $\|a\|_{L^q}\leq |B(x_0,r)|^{1/q-1}$,

iii) if $r\leq \frac{1}{\mathcal C_L}\rho(x_0)$ then $\int_{\mathbb R^d}a(x)dx=0$.
\end{Definition}

Then, we have the following atomic characterization of $H^1_L(\mathbb R^d)$.

\begin{theorema}[see \cite{DZ}, Theorem 1.5]\label{DZ, Theorem 1.5}
Let $1<q\leq \infty$. A function $f$ is in $H^1_L(\mathbb R^d)$ if and only if it can be written as $f=\sum_j \lambda_j a_j$, where $a_j$ are $(H^1_L,q)$-atoms and $\sum_j |\lambda_j|<\infty$. Moreover, there exists a constant $C>0$ such that
$$\|f\|_{H^1_L}\leq \inf\left\{\sum_j |\lambda_j|: f=\sum_j \lambda_j a_j\right\}\leq C \|f\|_{H^1_L}.$$
\end{theorema}

Let $P(x)= (4\pi)^{-d/2} e^{-|x|^2/4}$ be the Gauss function. According to \cite{DZ},  the space $h^1_n(\mathbb R^d)$, $n\in\mathbb Z$,  denotes the space of all integrable functions $f$ such that
$$\mathcal  M_nf(x) =\sup_{0<t<2^{-n/2}} |P_{t}*f(x)| \in L^1(\mathbb R^d),$$
where $P_{t}(\cdot):=t^{-d}P(t^{-1}\cdot)$. The norm on $h^1_n(\mathbb R^d)$ is then defined by 
$$\|f\|_{h^1_n}:= \|\mathcal M_n f\|_{L^1}.$$

It was shown in \cite{Go} that the dual space of $h^1_n(\mathbb R^d)$ can be identified with $bmo_n(\mathbb R^d)$ the space of all locally integrable functions $f$ such that
$$\|f\|_{bmo_n}= \|f\|_{BMO}+ \sup_{x\in\mathbb R^d, 2^{-n/2}\leq r}\frac{1}{|B(x,r)|}\int_{B(x,r)} |f(y)|dy<\infty.$$

Here and in what follows, for a ball $B$ and a locally integrable function $f$, we denote by $f_B$ the average of $f$ on B. Following  Dafni \cite{Da2}, we define $vmo_n(\mathbb R^d)$ as the subspace of $bmo_n(\mathbb R^d)$ consisting of those $f$ such that
$$\lim_{\sigma\to 0}\left( \sup_{x\in\mathbb R^d, r<\sigma} \frac{1}{|B(x,r)|}\int_{B(x,r)}|f(y)- f_{B(x,r)}|dy\right)=0$$
and
$$\lim_{R\to \infty}\left( \sup_{  B(x,r)\cap B(0,R)=\emptyset,  r\geq 2^{-n/2}}\frac{1}{|B(x,r)|}\int_{B(x,r)}|f(y)|dy\right)=0.$$

Recall that $C^\infty_c(\mathbb R^d)$  is the space of all $C^\infty$-functions with compact support. Then, the following was established  by  Dafni \cite{Da2}.

\begin{theoremb}[see \cite{Da2}, Theorem 6 and Theorem 9]\label{Dafni 2}
Let $n\in \mathbb Z$. Then,

i) The space $vmo_n(\mathbb R^d)$ is the closure of $C^\infty_c(\mathbb R^d)$ in $bmo_n(\mathbb R^d)$.

ii) The dual of $vmo_n(\mathbb R^d)$ is the space $h^1_n(\mathbb R^d)$.

\end{theoremb}

Furthermore, the weak$^*$-convergence is true in $h^1_n(\mathbb R^d)$.

\begin{theoremc}[see \cite{Da2}, Theorem 11]
Let $n\in\mathbb Z$. Suppose that $\{f_j\}_{j\geq 1}$ is a bounded sequence in $h^1_n(\mathbb R^d)$, and that $f_j(x)\to f(x)$ for  almost every $x\in\mathbb R^d$. Then, $f\in h^1_n(\mathbb R^d)$ and $\{f_j\}_{j \geq 1}$ weak$^*$-converges to $f$, that is, for every $\varphi\in vmo_n(\mathbb R^d)$, we have
$$\lim_{j\to\infty} \int_{\mathbb R^d} f_j(x) \varphi(x)dx = \int_{\mathbb R^d} f(x) \varphi(x) dx.$$
\end{theoremc}

\section{Proof of Theorem \ref{the main theorem}}

We begin by recalling the following two lemmas due to \cite{DZ}. These two lemmas together with Proposition \ref{Shen, Lemma 1.4} play an important role in our study.

\begin{Lemma}[see \cite{DZ}, Lemma 2.3] \label{DZ, Lemma 2.3}
There exists a constant $C>0$ and a collection of balls $B_{n,k}= B(x_{n,k}, 2^{-n/2})$, $n\in\mathbb Z, k=1,2,...$, such that $x_{n,k}\in \mathcal B_n$, $\mathcal B_n\subset \bigcup_k B_{n,k}$, and 
$$card \, \{(n',k'): B(x_{n,k}, R 2^{-n/2})\cap B(x_{n',k'}, R 2^{-n/2})\ne \emptyset\}\leq R^C$$
 for all $n,k$ and $R\geq 2$.
\end{Lemma}

\begin{Lemma}[see \cite{DZ}, Lemma 2.5]\label{DZ, Lemma 2.5}
There are nonnegative $C^\infty$-functions $\psi_{n,k}$, $n\in\mathbb Z, k=1,2,...$, supported in the balls $B(x_{n,k}, 2^{1-n/2})$ such that
$$\sum_{n,k}\psi_{n,k}=1\quad\mbox{and}\quad \|\nabla \psi_{n,k}\|_{L^\infty}\leq C 2^{n/2}.$$
\end{Lemma}

The following corollary is useful, its proof follows directly from Lemma \ref{DZ, Lemma 2.3}. We omit the details here (see also Corollary 1 of \cite{DGMTZ}).

\begin{Corollary}\label{a consequence of DZ, Lemma 2.3}
i) Let $\mathbb K$ be a compact set. Then, there exists a finite set $\Gamma\subset\mathbb Z\times \mathbb Z^+$ such that $\mathbb K\cap B(x_{n,k},2^{1-n/2})=\emptyset$ whenever $(n,k)\notin \Gamma$.

ii) There exists a constant $C>0$ such that for every $x\in \mathbb R^d$,
$$card \; \{(n,k)\in \mathbb Z\times \mathbb Z^+: B(x_{n,k}, 2^{1-n/2})\cap B(x, 2\rho(x))\ne \emptyset\}\leq C.$$

iii) There exists a constant $C>0$ such that for every ball $B(x,r)$ with $\rho(x)\leq r$, we have
$$|B(x,r)|\leq \sum_{B(x_{n,k},2^{-n/2})\cap B(x,r)\ne \emptyset}|B(x_{n,k}, 2^{-n/2})|\leq C |B(x,r)|.$$
\end{Corollary}

The key point in the proof of Theorem \ref{the main theorem} is the following result that we will prove in the last section. 

\begin{Theorem}\label{density of C^infty in VMO_L}
The space $C^\infty_c(\mathbb R^d)$ is dense in the space $VMO_L(\mathbb R^d)$.
\end{Theorem}

To prove Theorem \ref{the main theorem}, we need also the following two lemmas.

\begin{Lemma}[see \cite{Ky2}, Lemma 6.5]\label{Ky2, Lemma 6.5}
Let $1<q\leq \infty$,  $n\in\mathbb Z$ and $x\in \mathcal B_n$. Suppose that $f\in h^1_n(\mathbb R^d)$ with supp $f\subset B(x, 2^{1-n/2})$. Then, there are $(H^1_L,q)$-atoms $a_j$ related to the balls $B(x_j,r_j)$ such that $ B(x_j,r_j)\subset B(x, 2^{2-n/2})$ and
$$f= \sum_{j=1}^\infty \lambda_j a_j, \quad \sum_{j=1}^\infty |\lambda_j|\leq C \|f\|_{h^1_n}$$
with a positive constant $C$ independent of $n$ and $f$.
\end{Lemma}

\begin{Lemma}[see (4.7) in \cite{DZ}]\label{DZ}\label{DZ, 4.7}
For every $f\in H^1_L(\mathbb R^d)$, we have
$$\sum_{n,k} \|\psi_{n,k}f\|_{h^1_n}\leq C \|f\|_{H^1_L}.$$
\end{Lemma}

Now, we are ready to give the proof of the main theorem.

\begin{proof}[\bf  Proof of Theorem \ref{the main theorem}]
By assumption, there exists $\mathscr M>0$ such that 
$$\|f_j\|_{H^1_L}\leq \mathscr M\;,  \quad \mbox{for all}\; j\geq 1.$$

 Let $(n,k)\in\mathbb Z\times \mathbb Z^+$. Then,  for almost every $x\in\mathbb R^d$, $\psi_{n,k}(x)f_j(x) \to \psi_{n,k}(x)f(x)$ since  $f_j(x)\to f(x)$. By Theorem C, this yields that $\psi_{n,k}f$ belongs to $ h^1_n(\mathbb R^d)$ and $\{\psi_{n,k}f_j\}_{j\geq 1}$  weak$^*$-converges to $\psi_{n,k}f$ in $h^1_n(\mathbb R^d)$, that is,
\begin{equation}\label{weak convergence 1}
\lim\limits_{j\to\infty}\int_{\mathbb R^d}\psi_{n,k}(x)f_j(x)\phi(x)dx = \int_{\mathbb R^d} \psi_{n,k}(x)f(x)\phi(x)dx,
\end{equation}
for all $\phi\in C^\infty_c(\mathbb R^d)$. Furthermore, 
\begin{equation}\label{weak convergence 2}
\|\psi_{n,k}f\|_{h^1_n}\leq \varliminf\limits_{j\to\infty}\|\psi_{n,k}f_j\|_{h^1_n}.
\end{equation}

As $x_{n,k}\in\mathcal B_n$ and supp $\psi_{n,k}f\subset B(x_{n,k}, 2^{1-n/2})$, by Lemma \ref{Ky2, Lemma 6.5}, there are $(H^1_L,2)$-atoms $a_j^{n,k}$ related to the balls $B(x_j^{n,k}, r_j^{n,k})\subset B(x_{n,k}, 2^{2-n/2})$ such that 
$$\psi_{n,k}f =\sum_j \lambda_j^{n,k} a_j^{n,k}, \quad \sum_j |\lambda_j^{n,k}|\leq C \|\psi_{n,k}f\|_{h^1_n}.$$

Let $N,K\in\mathbb Z^+$ be arbitrary. Then, the above  together with (\ref{weak convergence 2}) and Lemma \ref{DZ, 4.7} imply that there exists $m_{N,K}\in\mathbb Z^+$ such that
\begin{eqnarray*}
\sum_{n=-N}^N \sum_{k=1}^K \sum_j |\lambda_j^{n,k}| &\leq& \sum_{n=-N}^N \sum_{k=1}^K C \Big(\frac{\mathscr M}{(1+n^2)(1+k^2)}+ \|\psi_{n,k}f_{m_{N,K}}\|_{h^1_n}\Big)\\
&\leq& C \sum_{n,k}\frac{\mathscr M}{(1+n^2)(1+k^2)}+ C \|f_{m_{N,K}}\|_{H^1_L}\\
&\leq& C \mathscr M,
\end{eqnarray*}
where the constants $C$ are independent of $N,K$. By Theorem A, this allows to conclude that $$f=\sum_{n,k}\psi_{n,k}f\in H^1_L(\mathbb R^d)\quad\mbox{and}\quad \|f\|_{H^1_L}\leq \sum_{n,k}\sum_j |\lambda_j^{n,k}| \leq C \mathscr M.$$

Finally, we need to show that for every $\phi\in VMO_L(\mathbb R^d)$, 
\begin{equation}\label{the last equation}
\lim_{j\to\infty} \int_{\mathbb R^d} f_j(x)\phi(x) dx= \int_{\mathbb R^d} f(x) \phi(x) dx.
\end{equation}

By Theorem \ref{density of C^infty in VMO_L}, we only need to prove (\ref{the last equation}) for $\phi\in C^\infty_c(\mathbb R^d)$. In fact, by $(i)$ of Corollary \ref{a consequence of DZ, Lemma 2.3}, there exists a finite set $\Gamma_\phi\subset \mathbb Z\times \mathbb Z^+$ such that 
$$f\phi = \sum_{(n,k)\in\Gamma_\phi} \psi_{n,k}f\phi\quad\mbox{and}\quad f_j\phi = \sum_{(n,k)\in\Gamma_\phi} \psi_{n,k}f_j\phi$$
 since supp $\psi_{n,k}\subset B(x_{n,k},2^{1-n/2})$. This together with (\ref{weak convergence 1}) give
\begin{eqnarray*}
\lim_{j\to\infty} \int_{\mathbb R^d} f_j(x)\phi(x) dx &=& \lim_{j\to\infty}\int_{\mathbb R^d}\sum_{(n,k)\in\Gamma_\phi} \psi_{n,k}(x)f_j(x) \phi(x) dx\\
&=&\sum_{(n,k)\in\Gamma_\phi} \lim_{j\to\infty}\int_{\mathbb R^d}\psi_{n,k}(x)f_j(x) \phi(x) dx\\
&=& \sum_{(n,k)\in\Gamma_\phi} \int_{\mathbb R^d}\psi_{n,k}(x)f(x) \phi(x) dx\\
&=& \int_{\mathbb R^d} f(x) \phi(x) dx,
\end{eqnarray*}
which ends the proof of Theorem \ref{the main theorem}.

\end{proof}

\section{Proof of Theorem \ref{density of C^infty in VMO_L}}\label{the last section}

The main point in the proof of Theorem \ref{density of C^infty in VMO_L} is the theorem.

\begin{Theorem}\label{the duality between CMO_L and H^1_L}
Let $CMO_L(\mathbb R^d)$ be the closure of $C^\infty_c(\mathbb R^d)$ in $BMO_L(\mathbb R^d)$. Then, $H^1_L(\mathbb R^d)$ is the dual space of $CMO_L(\mathbb R^d)$.
\end{Theorem}

To prove Theorem \ref{the duality between CMO_L and H^1_L}, we need the following three lemmas.

\begin{Lemma}\label{a consequence of Shen, Lemma 1.4}
There exists a constant $C>0$ such that 
$$2^{-n/2}\leq C r$$
whenever $B(x_{n,k},2^{1-n/2})\cap B(x,r)\ne \emptyset$ and $\rho(x)\leq r$.
\end{Lemma}

The proof of Lemma \ref{a consequence of Shen, Lemma 1.4} follows directly from Proposition \ref{Shen, Lemma 1.4}. We leave the details to the reader.

\begin{Lemma}\label{estimates for multipliers}
Let $\psi_{n,k}$, $(n,k)\in\mathbb Z\times \mathbb Z^+$, be as in Lemma \ref{DZ, Lemma 2.5}. Then, there exists a constant $C$ independent of $n,k,\psi_{n,k}$, such that 
\begin{equation}\label{estimates for multipliers 1}
\|\psi_{n,k}f\|_{bmo_n}\leq C \|f\|_{bmo_n}
\end{equation}
for all $f\in bmo_n(\mathbb R^d)$, and
\begin{equation}\label{estimates for multipliers 2}
\|\psi_{n,k}\phi\|_{BMO_L}\leq C \|\phi\|_{bmo_n}
\end{equation}
for all $\phi\in C^\infty_c(\mathbb R^d)$.
\end{Lemma}

\begin{Lemma}\label{equivalent norm for BMO_L}
For every $f\in BMO_L(\mathbb R^d)$, we have
$$\|f\|_{BMO_L}\approx \sup_{r\leq \rho(x)}\frac{1}{|B(x,r)|}\int_{B(x,r)} |f(y)- f_{B(x,r)}|dy+ \sup_{\rho(x)\leq r\leq 2\rho(x)} \frac{1}{|B(x,r)|}\int_{B(x,r)}|f(y)|dy.$$
\end{Lemma}

\begin{proof}[Proof of Lemma \ref{estimates for multipliers}]
Noting that $\psi_{n,k}$ is a multiplier of $bmo_n(\mathbb R^d)$ and $\|\psi_{n,k}\|_{L^\infty}\leq 1$, Theorem 2 of \cite{NY} allows us to reduce (\ref{estimates for multipliers 1}) to showing that
\begin{equation}\label{estimates for multipliers 3}
\frac{\log\Big(e+ \frac{2^{-n/2}}{r}\Big)}{|B(x,r)|}\int_{B(x,r)}\Big|\psi_{n,k}(y) -\frac{1}{|B(x,r)|}\int_{B(x,r)}\psi_{n,k}(z)dz\Big|dy\leq C
\end{equation}
holds for every ball $B(x,r)$ which satisfies $r\leq 2^{-n/2}$. In fact, from $\|\nabla \psi_{n,k}\|_{L^\infty}\leq C 2^{n/2}$ and the estimate  $\frac{r}{2^{-n/2}}\log\Big(e+ \frac{2^{-n/2}}{r}\Big)\leq \sup_{0<t\leq 1}t\log(e+1/t)<\infty$, 
\begin{eqnarray*}
&&\frac{\log\Big(e+ \frac{2^{-n/2}}{r}\Big)}{|B(x,r)|}\int_{B(x,r)}\Big|\psi_{n,k}(y) -\frac{1}{|B(x,r)|}\int_{B(x,r)}\psi_{n,k}(z)dz\Big|dy\\
&\leq& \frac{\log\Big(e+ \frac{2^{-n/2}}{r}\Big)}{|B(x,r)|^2}\int_{B(x,r)}\int_{B(x,r)}|\psi_{n,k}(y)-\psi_{n,k}(z)|dzdy\\
&\leq& \log\Big(e+ \frac{2^{-n/2}}{r}\Big) \|\nabla \psi_{n,k}\|_{L^\infty} 2r\\
&\leq& C  \frac{r}{2^{-n/2}}\log\Big(e+ \frac{2^{-n/2}}{r}\Big) \leq C,
\end{eqnarray*}
which proves (\ref{estimates for multipliers 3}), and thus (\ref{estimates for multipliers 1}) holds.

As (\ref{estimates for multipliers 1}) holds, we get
$$\|\psi_{n,k}\phi\|_{BMO}\leq \|\psi_{n,k}\phi\|_{bmo_n}\leq C \|\phi\|_{bmo_n}.$$
Therefore, to prove (\ref{estimates for multipliers 2}), we only need to show that 
\begin{equation}\label{estimates for multipliers 4}
\frac{1}{|B(x,r)|}\int_{B(x,r)}|\psi_{n,k}(y)\phi(y)|dy\leq C \|\phi\|_{bmo_n}
\end{equation}
holds for every $x\in \mathbb R^d$ and $r\geq \rho(x)$. Since supp $\psi_{n,k}\subset B(x_{n,k}, 2^{1-n/2})$,  (\ref{estimates for multipliers 4}) is obvious if $B(x,r)\cap B(x_{n,k}, 2^{1-n/2})=\emptyset$. Otherwise, as $\rho(x)\leq r$, Lemma \ref{a consequence of Shen, Lemma 1.4} gives $2^{-n/2}\leq C r$. As a consequence, we get
\begin{eqnarray*}
\frac{1}{|B(x,r)|}\int_{B(x,r)}|\psi_{n,k}(y) \phi(y)|dy &\leq& C \sup_{2^{-n/2}\leq r} \frac{1}{|B(x,r)|}\int_{B(x,r)}|\phi(y)|dy\\
&\leq& C \|\phi\|_{bmo_n},
\end{eqnarray*}
which proves (\ref{estimates for multipliers 4}), and hence (\ref{estimates for multipliers 2}) holds.

\end{proof}

\begin{proof}[Proof of Lemma \ref{equivalent norm for BMO_L}]
Clearly, it is sufficient to prove that
\begin{equation}\label{equivalent norm for BMO_L 0}
\sup_{\rho(x)\leq r} \frac{1}{|B(x,r)|}\int_{B(x,r)}|f(y)|dy\leq C \sup_{\rho(x)\leq r\leq 2\rho(x)} \frac{1}{|B(x,r)|}\int_{B(x,r)}|f(y)|dy.
\end{equation}

In fact, for every ball $B(x,r)$ which satisfies $\rho(x)\leq r$, setting
$$G=\{(n,k)\in \mathbb Z\times \mathbb Z^+: B(x_{n,k},2^{-n/2})\cap B(x,r)\ne \emptyset\},$$
one has 
$$B(x,r)\subset \cup_{(n,k)\in G}B(x_{n,k},2^{-n/2})\quad\mbox{and}\quad \sum_{(n,k)\in G}|B(x_{n,k}, 2^{-n/2})|\leq C |B(x,r)|$$
 since $\mathbb R^d=\cup_{n\in\mathbb Z}\mathcal B_n\subset \cup_{n,k}B(x_{n,k}, 2^{-n/2})$ and  $(iii)$ of Corollary \ref{a consequence of DZ, Lemma 2.3}. Therefore,
\begin{eqnarray*}
&&\frac{1}{|B(x,r)|}\int_{B(x,r)}|f(y)|dy \leq \frac{1}{|B(x,r)|} \int_{\cup_{(n,k)\in G}B(x_{n,k},2^{-n/2})} |f(y)|dy\\
&\leq&  \frac{1}{|B(x,r)|} \sum_{(n,k)\in G}|B(x_{n,k}, 2^{-n/2})| \sup_{\rho(z)\leq s\leq 2\rho(z)} \frac{1}{|B(z,s)|}\int_{B(z,s)}|f(y)|dy\\
&\leq& C \sup_{\rho(z)\leq s\leq 2\rho(z)} \frac{1}{|B(z,s)|}\int_{B(z,s)}|f(y)|dy,
\end{eqnarray*}
which implies that (\ref{equivalent norm for BMO_L 0}) holds.

\end{proof}

\begin{proof}[Proof of Theorem \ref{the duality between CMO_L and H^1_L}]
Since $CMO_L(\mathbb R^d)$ is a subspace of  $BMO_L(\mathbb R^d)$, which is the dual of $H^1_L(\mathbb R^d)$, every function $f$ in $H^1_L(\mathbb R^d)$ determines a bounded linear functional on $CMO_L(\mathbb R^d)$ of norm bounded by $\|f\|_{H^1_L}$.

Conversely, given a bounded linear functional $\mathcal T$ on $CMO_L(\mathbb R^d)$.  Then, for every $(n,k)\in\mathbb Z\times\mathbb Z^+$, from  (\ref{estimates for multipliers 2}) and density of $C^\infty_c(\mathbb R^d)$ in $vmo_n(\mathbb R^d)$,  the linear functional $\mathcal T_{n,k}(g)\mapsto \mathcal T(\psi_{n,k}g)$ is continuous on $vmo_n(\mathbb R^d)$. Consequently, by Theorem B, there exists $f_{n,k}\in h^1_n(\mathbb R^d)$ such that for every $\phi\in C^\infty_c(\mathbb R^d)$,
\begin{equation}\label{the duality between CMO_L and H^1_L 1}
\mathcal T(\psi_{n,k}\phi)= \mathcal T_{n,k}(\phi)=\int_{\mathbb R^d} f_{n,k}(y)\phi(y) dy,
\end{equation}
moreover, 
\begin{equation}\label{the duality between CMO_L and H^1_L 2}
\|f_{n,k}\|_{h^1_n}\leq C \|\mathcal T_{n,k}\|,
\end{equation}
where $C$ is a positive constant independent of $n,k,\psi_{n,k}$ and $\mathcal T$.

Noting that supp $\psi_{n,k}\subset B(x_{n,k}, 2^{1-n/2})$, (\ref{the duality between CMO_L and H^1_L 1}) implies that supp $f_{n,k}\subset B(x_{n,k}, 2^{1-n/2})$. Consequently, as $x_{n,k}\in\mathcal B_n$, Lemma \ref{Ky2, Lemma 6.5} yields that there are $(H^1_L,2)$-atoms $a_j^{n,k}$ related to the balls $B(x_j^{n,k},r_j^{n,k})$ such that 
\begin{equation}\label{the duality between CMO_L and H^1_L 3}
f_{n,k}= \sum_{j=1}^\infty \lambda_j^{n,k} a_j^{n,k}, \quad \sum_{j=1}^\infty |\lambda_j^{n,k}|\leq C \|f_{n,k}\|_{h^1_n}
\end{equation}
with a positive constant $C$ independent of $\psi_{n,k}$ and $f_{n,k}$.

Since supp $f_{n,k}\subset B(x_{n,k}, 2^{1-n/2})$, by Lemma \ref{DZ, Lemma 2.3}, the function 
$$x\mapsto f(x)=\sum_{n,k} f_{n,k}(x)$$
 is well defined, and belongs to $L^1_{\rm loc}(\mathbb R^d)$. Moreover, for every $\phi\in C^\infty_c(\mathbb R^d)$, by $(i)$ of Corollary \ref{a consequence of DZ, Lemma 2.3}, there exists a finite set $\Gamma_\phi\subset \mathbb Z\times \mathbb Z^+$ such that 
$$\mathcal T(\phi)= \sum_{(n,k)\in \Gamma_\phi} \mathcal T(\psi_{n,k} \phi)= \sum_{(n,k)\in \Gamma_\phi} \int_{\mathbb R^d}f_{n,k}(y)\phi(y) dy= \int_{\mathbb R^d} f(y) \phi(y)dy.$$

Next, we need to show that $f\in H^1_L(\mathbb R^d)$. 

We first claim that there exists $C>0$ such that
\begin{equation}\label{the duality between CMO_L and H^1_L 4}
\sum_{n,k}\|f_{n,k}\|_{h^1_n}\leq C \|\mathcal T\|.
\end{equation}
Assume that (\ref{the duality between CMO_L and H^1_L 4}) holds for a moment. Then, from (\ref{the duality between CMO_L and H^1_L 3}), there are $(H^1_L,2)$-atoms $a_j^{n,k}$ and complex numbers $\lambda_j^{n,k}$ such that
$$f=\sum_{n,k} \sum_j \lambda_j^{n,k} a_j^{n,k}\quad\mbox{and}\quad \sum_{n,k} \sum_j |\lambda_j^{n,k}|\leq C \sum_{n,k} \|f_{n,k}\|_{h^1_n}\leq C \|\mathcal T\|.$$
By Theorem A, this proves that $f\in H^1_L(\mathbb R^d)$, moreover, $\|f\|_{H^1_L}\leq C \|\mathcal T\|$. 

Now, we return to prove (\ref{the duality between CMO_L and H^1_L 4}). 

Without loss of generality, we can assume that $\mathcal T$ is a real-valued functional. By (\ref{the duality between CMO_L and H^1_L 2}), for each $(n,k)\in\mathbb Z\times \mathbb Z^+$,  there exists $\phi_{n,k}\in C^\infty_c(\mathbb R^d)$ such that 
\begin{equation}\label{the duality between CMO_L and H^1_L 5}
\|\phi_{n,k}\|_{vmo_n}\leq 1\quad\mbox{and}\quad \|f_{n,k}\|_{h^1_n}\leq C \mathcal T(\psi_{n,k} \phi_{n,k}).
\end{equation}

For any $\Gamma\subset \mathbb Z\times \mathbb Z^+$  a finite set, let $\phi =\sum_{(n,k)\in\Gamma}\psi_{n,k}\phi_{n,k}\in C^\infty_c(\mathbb R^d)$. We prove that $\|\phi\|_{BMO_L}\leq C$. Indeed, let  $B(x,r)$  be an arbitrary ball satisfying $r\leq 2\rho(x)$. Then, by $(ii)$ of Corollary \ref{a consequence of DZ, Lemma 2.3}, we get
$$card \,\{(n,k)\in\mathbb Z\times \mathbb Z^+: B(x_{n,k}, 2^{1-n/2})\cap B(x,r)\ne \emptyset\}\leq C.$$
This together with (\ref{estimates for multipliers 1}) and (\ref{the duality between CMO_L and H^1_L 5}) give
\begin{eqnarray*}
\frac{1}{|B(x,r)|}\int_{B(x,r)}|\phi(y)- \phi_{B(x,r)}|dy &\leq& C \sup_{(n,k)\in\Gamma}\|\psi_{n,k}\phi_{n,k}\|_{BMO} \\
&\leq& C \sup_{(n,k)\in\Gamma}\|\psi_{n,k}\phi_{n,k}\|_{bmo_n}\\
&\leq& C \sup_{(n,k)\in\Gamma} \|\phi_{n,k}\|_{bmo_n}\leq C
\end{eqnarray*}
if $r\leq \rho(x)$, and as (\ref{estimates for multipliers 4}), 
\begin{eqnarray*}
\frac{1}{|B(x,r)|}\int_{B(x,r)}|\phi(y)|dy &\leq& C \sup_{(n,k)\in\Gamma} \frac{1}{|B(x,r)|}\int_{B(x,r)}|\psi_{n,k}(y)\phi_{n,k}(y)|dy\\
&\leq& C \sup_{(n,k)\in\Gamma} \|\phi_{n,k}\|_{bmo_n}\leq C
\end{eqnarray*}
if $\rho(x)\leq r\leq 2\rho(x)$. Therefore, Lemma \ref{equivalent norm for BMO_L} yields
\begin{eqnarray*}
\|\phi\|_{BMO_L}&\leq& C \Big\{  \sup_{r\leq \rho(x)}\frac{1}{|B(x,r)|}\int_{B(x,r)}|\phi(y)- \phi_{B(x,r)}|dy +\\
&&\quad\quad+ \sup_{\rho(x)\leq r\leq 2\rho(x)}\frac{1}{|B(x,r)|}\int_{B(x,r)} |\phi(y)|dy\Big\}\\
&\leq& C
\end{eqnarray*}
since $B(x,r)$  is an arbitrary ball satisfying $r\leq 2\rho(x)$. This implies that
\begin{eqnarray*}
\sum_{(n,k)\in\Gamma} \|f_{n,k}\|_{h^1_n} &\leq& C\sum_{(n,k)\in\Gamma} \mathcal T(\psi_{n,k}\phi_{n,k})= C \mathcal T(\phi)\\
&\leq& C \|\mathcal T\|  \|\phi\|_{BMO_L}\leq C \|\mathcal T\|.
\end{eqnarray*}
Consequently, (\ref{the duality between CMO_L and H^1_L 4}) holds since $\Gamma\subset \mathbb Z\times \mathbb Z^+$  is an arbitrary finite set and the constants $C$ are dependent of $\Gamma$. This ends the proof of Theorem \ref{the duality between CMO_L and H^1_L}.

\end{proof}

To prove Theorem \ref{density of C^infty in VMO_L}, we need to recall the following lemma.

\begin{Lemma}[see \cite{DZ}, Lemma 3.0]\label{DZ, Lemma 3.0}
There is a constant $\varepsilon>0$ such that for every $C'$ there exists $C>0$ such that for every $t>0$ and $|x-y|\leq C'\rho(x)$,
$$\Big| \frac{1}{(4\pi t)^{d/2}} e^{-\frac{|x-y|^2}{4t}} - T_t(x,y)\Big|\leq C \frac{1}{|x-y|^d}\Big(\frac{|x-y|}{\rho(x)}\Big)^\varepsilon.$$

\end{Lemma}

\begin{proof}[\bf Proof of Theorem \ref{density of C^infty in VMO_L}]
As $H^1_L(\mathbb R^d)$ is the dual space of $VMO_L(\mathbb R^d)$ (see Theorem 4.1 of \cite{DDSTY}), by Theorem \ref{the duality between CMO_L and H^1_L} and the Hahn-Banach theorem, it suffices to show that $C^\infty_c(\mathbb R^d)\subset  VMO_L(\mathbb R^d)$. In fact, for every $f\in C^\infty_c(\mathbb R^d)$ with supp $f\subset B(0, R_0)$, one only needs to establish the following three steps:

{\bf Step 1.} By (\ref{the Feynman-Kac formula}), one has $\|e^{-tL}f\|_{L^2}\leq \|f\|_{L^2}$ for all $t>0$. Therefore,
$$\frac{1}{|B(x,t)|}\int_{B(x,t)}|f(y)-  e^{-tL}f(y)|^2 dy\leq \frac{1}{|B(x,t)|}4 \|f\|_{L^2}^2$$
for all $x\in\mathbb R^d$ and $t>0$. This implies that 
$$\gamma_2(f)=\lim\limits_{R\to \infty}\left(\sup\limits_{x\in\mathbb R^d, t\geq R}\Big(\frac{1}{|B(x,t)|}\int_{B(x,t)}|f(y)- e^{-tL}f(y)|^2 dy\Big)^{1/2}\right)=0.$$

{\bf Step 2.} For every $R> 2 R_0$ and $B(x,t)\cap B(0,R)=\emptyset$, by (\ref{the Feynman-Kac formula}) again,
\begin{eqnarray*}
&&\frac{1}{|B(x,t)|}\int_{B(x,t)}|f(y)- e^{-tL}f(y)|^2 dy\\
&\leq& \frac{1}{|B(x,t)|} \int_{B(x,t)}\Big( \frac{1}{(4\pi t)^{d/2}}\int_{B(0,R_0)} e^{-\frac{(R- R_0)^2}{4t}}|f(z)|dz\Big)^2 dy\\
&\leq& (4\pi)^{-d} \|f\|^2_{L^1}\frac{1}{t^{d}} e^{- \frac{R^2}{8 t}} \leq (4\pi)^{-d} \|f\|^2_{L^1} \Big(\frac{8d}{R^2}\Big)^d e^{-d}.
\end{eqnarray*}
Therefore,
$$\gamma_3(f)=\lim\limits_{R\to \infty}\left(\sup\limits_{B(x,t)\cap B(0,R)=\emptyset}\Big(\frac{1}{|B(x,t)|}\int_{B(x,t)}|f(y)- e^{-tL}f(y)|^2 dy\Big)^{1/2}\right)=0.$$

{\bf Step 3.} Finally, we need to show that
\begin{equation}\label{density of C^infty in VMO_L 1}
\gamma_1(f)= \lim_{r\to 0}\left( \sup_{x\in\mathbb R^d, t\leq r}\Big(\frac{1}{|B(x,t)|}\int_{B(x,t)} |f(y) - e^{-tL}f(y)|^2dy \Big)^{1/2} \right)=0.
\end{equation}

For every $x\in \mathbb R^d$ and $t>0$, we have
\begin{eqnarray*}
&&\left\{\frac{1}{|B(x,t)|}\int_{B(x,t)} \Big|f(y) - \frac{1}{(4\pi t)^{d/2}}\int_{\mathbb R^d} e^{-\frac{|y-z|^2}{4t}}f(z)dz\Big|^2 dy\right\}^{1/2}\\
&\leq& \sup_{|y-z|< t^{1/4}}|f(y) - f(z)| + 2 \|f\|_{L^\infty} \frac{1}{(4\pi t)^{d/2}} \int_{|z|\geq t^{1/4}} e^{-\frac{|z|^2}{4t}}dz.
\end{eqnarray*}

By the uniformly continuity of $f$, the above implies that
$$\lim_{r\to 0}\left( \sup_{x\in\mathbb R^d, t\leq r}\Big(\frac{1}{|B(x,t)|}\int_{B(x,t)} \Big|f(y) - \frac{1}{(4\pi t)^{d/2}}\int_{\mathbb R^d} e^{-\frac{|y-z|^2}{4t}}f(z)dz\Big|^2dy \Big)^{1/2} \right)=0.$$
Therefore, we can reduce (\ref{density of C^infty in VMO_L 1}) to showing that
\begin{equation}\label{density of C^infty in VMO_L 2}
\lim_{r\to 0}\left( \sup_{x\in\mathbb R^d, t\leq r}\Big(\frac{1}{|B(x,t)|}\int_{B(x,t)} \Big[\int_{\mathbb R^d} \Big|\frac{1}{(4\pi t)^{d/2}} e^{-\frac{|y-z|^2}{4t}}- T_t(y,z)\Big| |f(z)|dz\Big]^2dy \Big)^{1/2} \right)=0.
\end{equation}

From supp $f\subset B(0,R_0)$ and $\mathbb R^d\equiv \cup_{n,k} B(x_{n,k}, 2^{-n/2})$, there exists a finite set $\Gamma_f\subset \mathbb Z\times\mathbb Z^+$ such that supp $f\subset \cup_{(n,k)\in\Gamma_f} B(x_{n,k}, 2^{-n/2})$. As a consequence, (\ref{density of C^infty in VMO_L 2}) holds if we can prove that for each $(n,k)\in \Gamma_f$,
\begin{equation}\label{density of C^infty in VMO_L 3}
\lim_{r\to 0}\left( \sup_{x\in\mathbb R^d, t\leq r}\Big(\frac{1}{|B(x,t)|}\int_{B(x,t)} \Big[\int_{B(x_{n,k},2^{-n/2})} \Big|\frac{1}{(4\pi t)^{d/2}} e^{-\frac{|y-z|^2}{4t}}- T_t(y,z)\Big| |f(z)|dz\Big]^2dy \Big)^{1/2} \right)=0.
\end{equation}

We now prove (\ref{density of C^infty in VMO_L 3}). Let $x\in\mathbb R^d$ and $0<t< 2^{-2n}$. As $x_{n,k}\in \mathcal B_n$, by Proposition \ref{Shen, Lemma 1.4}, there is a constant $C>1$ such that $C^{-1} 2^{-n/2}\leq \rho(z)\leq C 2^{-n/2}$ for all $z\in B(x_{n,k},2^{-n/2})$. This together with (\ref{the Feynman-Kac formula}) and Lemma \ref{DZ, Lemma 3.0}, give
\begin{eqnarray*}
&&\left\{\frac{1}{|B(x,t)|}\int_{B(x,t)} \Big[\int_{B(x_{n,k},2^{-n/2})} \Big|\frac{1}{(4\pi t)^{d/2}} e^{-\frac{|y-z|^2}{4t}}- T_t(y,z)\Big| |f(z)|dz\Big]^2dy\right\}^{1/2}\\
&\leq&  2\|f\|_{L^\infty} \frac{1}{(4\pi t)^{d/2}} \int_{|z|\geq t^{1/4}} e^{-\frac{|z|^2}{4t}}dz + C 2^{ n\varepsilon/2} \|f\|_{L^\infty} \int_{|z|< t^{1/4}}\frac{1}{|z|^{d-\varepsilon}}dz,
\end{eqnarray*}
which implies that (\ref{density of C^infty in VMO_L 3}) holds. The proof of Theorem \ref{density of C^infty in VMO_L} is thus completed.

\end{proof}

\medskip
\noindent 
Department of Mathematics, University of Quy Nhon, 170 An Duong Vuong, Quy Nhon, Binh Dinh, Viet Nam \\
Email: dangky@math.cnrs.fr

\end{document}